\documentclass{article}

\usepackage[margin=1.0in]{geometry}
\usepackage{amsmath}
\usepackage{amssymb}
\usepackage{amsthm}
\usepackage{enumerate}

\makeatletter
\newtheorem*{rep@theorem}{\rep@title}
\newcommand{\newreptheorem}[2]{%
\newenvironment{rep#1}[1]{%
 \def\rep@title{#2 \ref{##1}}%
 \begin{rep@theorem}}%
 {\end{rep@theorem}}}
\makeatother

\newtheorem{lem}{Lemma}
\newtheorem{thm}{Theorem}
\newtheorem{dfn}{Definition}

\newtheorem*{prb*}{Problem}
\newtheorem*{lem*}{Lemma}
\newreptheorem{lem}{Lemma}
\providecommand{\e}{\varepsilon}

\providecommand{\f}{\varphi}

\providecommand{\x}{\xi}

\providecommand{\R}{\mathbb{R}}

\providecommand{\Z}{\mathbb{Z}}
\providecommand{\N}{\mathbb{N}}

\providecommand{\D}{\mathbb{D}}

\title{Areas spanned by point configurations in the plane}
\author{Alex McDonald \\ University of Rochester \\ Rochester, NY}

\begin{document}

\maketitle

\begin{abstract}
We consider an over-determined Falconer type problem on $(k+1)$-point configurations in the plane using the group action framework introduced in \cite{GroupAction}.  We define the area type of a $(k+1)$-point configuration in the plane to be the vector in $\R^{\binom{k+1}{2}}$ with entries given by the areas of parallelograms spanned by each pair of points in the configuration.  We show that the space of all area types is $2k-1$ dimensional, and prove that a compact set $E\subset\R^d$ of sufficiently large Hausdorff dimension determines a positve measure set of area types.
\end{abstract}

\section{Introduction}

One of the most important and interesting problems in geometric measure theory is the Falconer distance problem, first posed by Falconer in 1986 \cite{Falconer}.  Given a set $E\subset\R^d$, define its distance set to be $\Delta(E):=\{\|x-y\|:x,y\in E\}$.  The Falconer distance problem asks how large the Hausdorff dimension of a compact set $E$ must be to ensure the distance set $\Delta(E)$ has positive Lebesgue measure.  Falconer proved that $|\Delta(E)|>0$ if $\dim E>\frac{d+1}{2}$, and showed that there were sets with dimension arbitrarily close to but less that $\frac{d}{2}$ such that $|\Delta(E)|=0$, implying that the correct threshold is in the range $\left[\frac{d}{2},\frac{d+1}{2}\right]$.  Falconer conjectured that the correct threshold was in fact $\frac{d}{2}$.  \\

In 1987 Mattila \cite{Mattila} introduced a framework which became fundamental to the study of the Falconer problem.  Given a measure $\mu$, he considered the quantity

\[
\mathcal{M}(\mu):=\int\left(\int_{S^{d-1}}|\widehat{\mu}(t\omega)|^2\:d\omega\right)^2t^{d-1}\: dt
\]

now known as the Mattila integral.  He proved that if $\mu$ is a measure supported on $E$ with finite energy integral and $\mathcal{M}(\mu)<\infty$, then $|\Delta(E)|>0$.  This insight was used by Wolff \cite{WolffExp} in 1999 to prove that if $E\subset\R^2$ is compact with $\dim E>\frac{4}{3}$ then $|\Delta(E)|>0$, and by Erdogan \cite{ErdoganExp} in 2006 to obtain the threshold $\frac{d}{2}+\frac{1}{3}$ in dimension $d$, matching Wolff in the $d=2$ case and improving Falconer's result when $d\geq 3$.  These were the best known results until recently, when the decoupling theorem of Bourgain and Demeter \cite{Decoupling} provided a new tool to attack the problem.  This led to several improvements in quick succession.  The best results currently state that for compact $E\subset \R^d$, the distance set $\Delta(E)$ has positive Lebesgue measure if $\dim E>s_d$ where

\[
s_d=
\begin{cases}
5/4, & d=2 \  \cite{GIOW} \\
9/5, & d=3  \  \cite{GuthEtAl} \\
\frac{d}{2}+\frac{1}{4}, & d\geq 4, d \text{ even } \  \cite{evendim} \\
\frac{d}{2}+\frac{1}{4}+\frac{1}{4(d-1)}, & d\geq 4, d \text{ odd } \  \cite{DZ}
\end{cases}
\]

A major breakthrough occured in the discrete version of the Falconer problem, known as the Erdos distinct distance problem, when Elekes and Sharir \cite{ES} observed that the problem could be studied in terms of the action of the isometry group, since two pairs of points have the same distance if and only if there exists an isometry mapping one pair to the other.  Guth and Katz \cite{GK} solved the Erdos distinct distance problem in the plane using this framework.  Greenleaf, Iosevich, Liu, and Palsson \cite{GroupAction} then applied this framework to the study of the Falconer problem.  They observed the original Mattila integral $\mathcal{M}(\mu)$ was a constant multiple of the integral

\[
\int\int |\widehat{\mu}(\x)|^2|\widehat{\mu}(g\x)|^2\:dg\:d\x,
\]

where $dg$ denotes a Haar measure on the orthogonal group $O(\R^d)$.  Using this observation, they give a new proof of Falconer's $\frac{d+1}{2}$ threshold in terms of this group action.  More generally, for any $k\leq d$ two $(k+1)$-point configurations will be congruent (in the sense that all pairwise distances between points are the same) if and only if there is a rigid motion which takes one configuration to the other.  Using this observation, the authors consider the following generalization of Falconer's problem.  Let $\Delta_k(E)$ be the set of congruence classes of $(k+1)$-point configurations of points in $E$ for $k\leq d$.  There are $\binom{k+1}{2}$ pairwise distances, so the space of all possible congruence classes can be identified with $\R^{\binom{k+1}{2}}$.  The authors prove that if $\dim E\geq d-\frac{d-1}{k+1}$ then $\Delta_k(E)$ has positive $\binom{k+1}{2}$ dimensional Lebesgue measure.  \\

The results of \cite{GroupAction} were extended by Chatzikonstantinou, Iosevich, Mkrtchyan, and Pakianathan in \cite{Rigidity}, where the authors establish a non-trivial dimensional threshold in the $k>d$ case as well.  The main obstacle here is that the system becomes over-determined; when $k\leq d$ one must specify all $\binom{k+1}{2}$ distances to determine a congruence class, whereas when $k>d$ one only needs to specify some distances.  As a result, the space of congruence classes is no longer $\R^{\binom{k+1}{2}}$, but rather a lower dimensional subset.  The authors prove the space of congruence classes can be identified with a space of dimension $d(k+1)-\binom{d+1}{2}$ with a natural measure, and use the group action framework to prove that compact sets $E\subset \R^d$ of dimension greater than $d-\frac{1}{k+1}$ determine a positive measure set of congruence classes.\\

In this paper, we study another over-determined point configuration problem using the group action framework.  Specifically, given a $(k+1)$-point configuration, instead of looking at pairwise distances we look at the areas of the parallelograms spanned by pairs of points in our configuration.  More precisely, we have the following definition

\begin{dfn}
Given a $(k+1)$-point configuration $x=(x^1,...,x^{k+1})\in (\R^2)^{k+1}$, the \textbf{area type} of $x$ is the vector

\[
(x^i\cdot x^{j\perp})_{1\leq i<j\leq k+1}\in\R^{\binom{k+1}{2}}.
\]

Given a set $E\subset\R^2$, the set of area types determined by $(k+1)$-point configurations in $E$ is 

\[
\mathcal{A}_k(E)=\{(x^i\cdot x^{j\perp})_{1\leq i<j\leq k+1}:x \in E^{k+1}\}
\]

Denote $\mathcal{A}_k(\R^2)$ by $\mathcal{A}_k$.
\end{dfn}

It is easy to see that a configuration has area type $\textbf{0}$ if and only if all points of the configuration lie on a common line through the origin.  Other than this degenerate case, we will see that two configurations $x$ and $y$ have the same area type if and only if there exists a unique $g\in \text{SL}_2(\R)$ such that $y^i=gx^i$ for all $i$.  Thus, this notion of equivalence is (outside a negligible class of degenerate configurations) consistent with the group action of $\text{SL}_2(\R)$.  \\

In \cite{Liu} Liu develops a general framework for problems concerning orbits of point configurations under group actions.  In Liu's setup, one considers a group $G$ acting on $\R^d$ and a smooth map $\Phi:(\R^d)^{k+1}\to\R^m$ satisfying certain conditions, and with the property that $\Phi(x^1,\cdots ,x^{k+1})=\Phi(y^1,\cdots,y^{k+1})$ if and only if there exists $g\in G$ such that $y^i=gx^i$ for each $i$.  Under these assumptions, Liu establishes a sufficient condition on compact sets $E\subset\R^d$ to ensure $\Delta_\Phi:=\{\Phi(x):x\in E^{k+1}\}\subset \R^m$ has positive $m$-dimensional Lebesgue measure.  However, the conditions assumed on $\Phi$ cannot be met if the image of $\Phi$ is contained in a lower dimensional submanifold of $\R^m$, as is the case in the over-determined problem considered here (it is also clear that the conclusion of Liu's theorem couldn't possibly hold in this case).  \\

Liu considers the action of $\text{SL}_2(\R)$ in the case $k=1$ (\cite{Liu}, Theorem 1.4).  We shall see that the space of non-degenerate area types has dimension $2k-1$, so this problem becomes over-determined for $k>2$.  In this case, we will see that $\mathcal{A}_k$ can be identified with a $(2k-1)$-flat in $\R^{2k+2}$ which we equip with the $2k-1$ dimensional Lebesgue measure $\mathcal{L}_{2k-1}$.  The set $\mathcal{A}_k(E)$ can then be viewed as a subset of that space.  With this setup, our main result is as follows.

\begin{thm}
\label{main}
Let $E\subset\R^2$ be a compact set with $\dim E>2-\frac{1}{2k}$.  Then, $\mathcal{L}_{2k-1}(\mathcal{A}_k(E))>0$.
\end{thm}

We note that throughout, we are working with the signed area $x\cdot y^\perp$ rather than unsigned area $|x\cdot y^\perp|$.  It is clear by the pigeonhole principle that the above theorem applies to unsigned area types as well. \\

Lastly, we want to discuss the extent to which the above theorem is sharp.  In \cite{Rigidity} the authors obtained the slightly better threshold $2-\frac{1}{k+1}$ in the analogous problem for congruence classes of $(k+1)$-point configurations.  While the translation invariance in that problem allowed some arguments that don't work in our present setting, it is nevertheless reasonable to conjecture that the $2-\frac{1}{k+1}$ threshold should be true here as well. \\

We can, however, say that the correct threshold must tend to $2$ as $k\to\infty$.  In particular, we have the following result.

\begin{thm}[Sharpness]
\label{Sharpness}
Let $k\in\N$ and let $s<2-\frac{4}{2k+1}$.  There exists a compact set $E\subset\R^2$ with $\dim E=s$ and $\mathcal{L}_{2k-1}(\mathcal{A}_k(E))=0$.
\end{thm}

\subsection{Acknowledgments}

This paper is part of the author's Ph.D. thesis.  The author would like to thank Alex Iosevich for suggesting the problem and for being a constant source of advice and encouragement.

\section{The group action viewpoint}

The notion of area type defined in the introduction gives an equivalence relation on the set of $(k+1)$-point configurations, where two configurations are equivalent if they have the same area type.  One of the key observations we use to prove the main theorem is that for typical configurations, the equivalence classes under this relation are the orbits of the natural action of the group $\text{SL}_2(\R)$ on the space of configurations.  More precisely, we have the following.

\begin{dfn}
A $(k+1)$-point configuration $x$ is called \textbf{degenerate} if $\{x^1,x^2\}$ is linearly dependent, and \textbf{non-degenerate} otherwise.  Similarly, $x$ is called $c$-degenerate for $c>0$ if $|x^1\cdot x^{2\perp}|<c$, and $c$-non-degenerate otherwise.
\end{dfn}

Since a $c$-non-degenerate configuration is also $c'$-non-degenerate for $c'<c$, we will always assume $c<1$ when $c$-non-degeneracy is assumed.

\begin{lem}
\label{action}
Let $x,y\in(\R^2)^{k+1}$ be non-degenerate configurations.  Then $x$ and $y$ have the same area type if and only if there is a unique $g\in\text{SL}_2(\R)$ such that for each $i$, we have $y^i=gx^i$.
\end{lem} 

\begin{proof}
First, suppose $x$ and $y$ have the same area types.  Because $x$ is non-degenerate, $x^1\cdot x^{2\perp}\neq 0$.  Equivalently, the $2\times 2$ matrix with columns $x^1,x^2$ is non-singular; we denote this matrix by $(x^1\ x^2)$.  Let

\[
g=(y^1\ y^2)(x^1\ x^2)^{-1}.
\]

Note that $g(x^1\ x^2)=(gx^1\ gx^2)$, so this implies $gx^1=y^1$ and $gx^2=y^2$.  Let $i$ be any index, and write

\[
x^i=ax^1+bx^2,\ y^i=a'y^1+b'y^2.
\]

We have

\[
x^1\cdot x^{i\perp}=x^1\cdot (ax^1+bx^2)^\perp=x^1\cdot (ax^{1\perp}+bx^{2\perp})=bx^1\cdot x^{2\perp}.
\]

Similarly, $y^1\cdot y^{i\perp}=b'y^1\cdot y^{2\perp}$.  By assumption, $x^1\cdot x^{i\perp}=y^1\cdot y^{i\perp}$, so it follows that $bx^1\cdot x^{2\perp}=b'y^1\cdot y^{2\perp}$.  Since $x^1\cdot x^{2\perp}=y^1\cdot y^{2\perp}\neq 0$, we conclude $b=b'$.  An argument considering $x^2\cdot x^{i\perp}$ similarly shows that $a=a'$.  So,

\[
gx^i=agx^1+bgx^2=ay^1+by^2=y^i.
\]

This proves existence.  Uniqueness follows from the fact that the configuration contains a basis, so $g$ is determined by its action on the configuration.  The converse follows from the matrix equation $g(x^i\ x^j)=(y^i\ y^j)$ and the fact that $g$ has determinant 1.
\end{proof}

\textbf{Remark}.  Clearly, the same argument can be made if any two indices $i,j$ have $\{x^i,x^j\}$ independent, and this will always happen unless all $x^i$ are on a common line through the origin.  However, either definition will produce a measure zero set of degenerate configurations, and it is convenient to assume that the first two points are always independent. \\

The lemma shows that we can identify the non-degenerate area types of $\mathcal{A}_k$ with the $(2k-1)$ flat in $\R^{2k+2}$ consisting of points where the first three coordinates are $(1,0,0,...)$.  Given $x\in(\R^2)^{k+1}$ non-degenerate, there is a unique $g\in\text{SL}_2(\R)$with $gx^1=(1,0)$ and $gx^2=(0,x^1\cdot x^{2\perp})$.  For this choice of $g$, let 

\[
A(x)=(gx^1,...,gx^{k+1})=(1,0,0,x^1\cdot x^{2\perp},gx^3,...,gx^{k+1}).
\]

Then $A(x)$ can be viewed as a point in the aforementioned $(2k-1)$-flat, or as a $(k+1)$ point configuration with the same area type as $x$.  It also follows that $A(x)=A(y)$ if and only if $x$ and $y$ have the same area type.  We also have the following approximate version of this fact.

\begin{lem}
\label{approx}
Let $x,y$ be $c$-non-degenerate $(k+1)$-point configurations satisfying $\|x^i\|,\|y^i\|\leq 1$ for each $i$.  Then, for $0<\e<1$, we have

\begin{enumerate}[i)]
	\item If $|x^i\cdot x^{j\perp}-y^i\cdot y^{j\perp}|<\e$ for all $i,j$, then $\|A(x)-A(y)\|<\frac{\sqrt{5k}}{c}\e$.
	\item If $\|A(x)-A(y)\|<\e$, then $|x^i\cdot x^{j\perp}-y^i\cdot y^{j\perp}|<6\e$ for all $i,j$.
\end{enumerate}
\end{lem}

\begin{proof}
To prove the first part, suppose $|x^i\cdot x^{j\perp}-y^i\cdot y^{j\perp}|<\e$ for all $i,j$.  Let $A(x)=t=(1,0,0,t_1,...,t_{2k-1})$ and $A(y)=s=(1,0,0,s_1,...,s_{2k-1})$.  Then our assumption with $i=1$ means $t$ and $s$ satisfy

\[
|(1,0)\cdot (t_{2j},t_{2j+1})^\perp-(1,0)\cdot (s_{2j},s_{2j+1})^\perp|<\e
\]

for each $1\leq j\leq k-1$, which reduces to $|t_{2j+1}-s_{2j+1}|<\e$.  We also have $|t_1-s_1|<\e$.  Finally, we have

\[
|(t_{2j},t_{2j+1})\cdot (0,t_1)^\perp-(s_{2j},s_{2j+1})\cdot (0,s_1)^\perp|<\e,
\]

which reduces to $|t_{2j}t_1-s_{2j}s_1|<\e$, so

\begin{align*}
|t_{2j}-s_{2j}|&\leq \frac{1}{c} |(t_{2j}-s_{2j})t_1| \\
&\leq \frac{1}{c}|t_{2j}t_1-s_{2j}s_1|+\frac{1}{c}|s_{2j}(t_1-s_1)| \\
&\leq \frac{2\e}{c}.
\end{align*}

Putting this together, we have

\begin{align*}
\|A(x)-A(y)\|^2&=(t_1-s_1)^2+\sum_{j=1}^{k-1} (t_{2j}-s_{2j})^2+(t_{2j+1}-s_{2j+1})^2 \\
&\leq\e^2+\sum_{j=1}^{k-1}(\e^2+\frac{4\e^2}{c^2}) \\
&=k\e^2+\frac{4(k-1)\e^2}{c^2} \\
&\leq \frac{5k\e^2}{c^2},
\end{align*}

or $\|A(x)-A(y)\|\leq\frac{\sqrt{5k}}{c}\e$ as claimed. \\

For the second part, let $A(x)=t$ and $A(y)=s$ with the same notation as above.  If $\|A(x)-A(y)\|<\e$, then $|t_i-s_i|<\e$ for each $i$.  Write

\[
(s_{2i},s_{2i+1},s_{2j},s_{2j+1})=(t_{2i}+\delta_1,t_{2i+1}+\delta_2,t_{2j}+\delta_3,t_{2j+1}+\delta_4)
\]

with $\delta_1,\delta_2,\delta_3,\delta_4<\e$.  We have

\begin{align*}
&|(t_{2i},t_{2i+1})\cdot (t_{2j},t_{2j+1})^\perp-(s_{2i},s_{2i+1})\cdot (s_{2j},s_{2j+1})^\perp| \\
=&|t_{2i}\delta_4+t_{2j+1}\delta_1+\delta_1\delta_4-t_{2i+1}\delta_3-t_{2j}\delta_2-\delta_2\delta_3| \\
\leq & 6\e.
\end{align*}
\end{proof}

Our final tool is the following change of variables formula.

\begin{thm}
\label{gsub}
Let $x$ be a $c$-non-degenerate $(k+1)$-point configuration with $\|x^i\|\leq 1$ for all $i$.  Let $dy$ denote integration with respect to the $2k+2$ dimensional Lebesgue measure restricted to the domain $\|y^i\|\leq 1$, and let $dg$ denote integration with respect to a Haar measure on $\text{SL}_2(\R)$.  For any integrable $f$, we have

\begin{align*}
&\lim_{\e\to 0}\e^{-(2k-1)}\int_{|x^i\cdot x^{j\perp}-y^i\cdot y^{j\perp}|<\e}f(y)\:dy \\
\approx&\lim_{\e\to 0}\e^{-(2k-1)}\int_{|A(x)-A(y)|<\e}f(y)\:dy \\
\approx&\int f(gx)\:dg,
\end{align*}

where the implicit constants depend on $k$ and $c$, if the limit exists.  If the limit does not exist, the statement holds with the limit replaced by lim sup or lim inf. 
\end{thm}

\begin{proof}
Let $\D$ denote the unit disk in $\R^2$.  Define

\begin{align*}
S_1(\e)&=\{y\in\D^{k+1}:\forall i,j\ |x^i\cdot x^{j\perp}-y^i\cdot y^{j\perp}|<\e\} \\
S_2(\e)&=\{y\in\D^{k+1}:\|A(x)-A(y)\|<\e\} \\
S_3(\e)&=\{y\in\D^{k+1}:\exists g\in\text{SL}_2(\R)\|y-gx\|<\e\}
\end{align*}

where in the third definition we let $g$ act on $x$ pointwise and view the result as a vector in $\R^{2k+2}$ with the usual norm.  Define corresponding averaging operators

\[
I_{n,\e}(f)=\frac{1}{\mathcal{L}_{2k+2}(S_n(\e))}\int_{S_n(\e)}|f|
\]

for $n=1,2,3$.  Our first observation is that $\mathcal{L}_{2k+2}(S_n(\e))\approx \e^{2k-1}$ for each $n$.  Moreover, for each pair $n,m=1,2,3$ there is a constant $C_{n,m}$, possibly depending on $c$ and $k$, such that $S_n(\e)\subset S_m(C_{n,m}\e)$.  This has already been proved in the previous lemma for $n,m=1,2$.  Since there exist $g,h$ such that $gx=A(x)$ and $hy=A(y)$, we have

\[
\|h^{-1}gx-y\|=\|h^{-1}(A(x)-A(y))\|\leq \|h^{-1}\|\|A(x)-A(y)\|.
\]

Checking the construction of $h$ in the proof of lemma \ref{action}, we see $\|h^{-1}\|\leq \frac{2\sqrt{2}}{c}$.  This shows the existence of $C_{2,3}$.  To show the existence of $C_{3,1}$ consider $y\in S_3(\e)$.  It follows that $\|y^i-gx^i\|<\e$ for each $i$.  This means for each $i$, there is a $\delta^i=(\delta_1^i,\delta_2^i)$ with $y^i=gx^i+\delta^i$ and $\|\delta^i\|<\e$.  Using this, and the fact that $x^i\cdot x^{j\perp}=gx^i\cdot gx^{j\perp}$, we can take $C_{3,1}=6$.  This shows that all $C_{n,m}$ exist.  It follows that $I_{1,\e}(f)\approx I_{2,\e}(f)\approx I_{3,\e}(f)$.  Since $I_{3,\e}(f)$ approximates $\int f(gx)\:dg$ as $\e\to 0$, this completes the proof.
\end{proof}

\section{Proofs}

\subsection{Initial reductions and outline of proof}

We start by making a couple simple reductions that allow us to apply the results of the previous section.  In order to freely use Lemma \ref{gsub}, we need to ensure that the configurations we are integrating over satisfy $\|x^i\|,\|y^i\|\leq 1$ and $|x^i\cdot x^{j\perp}|>c$ for some constant $c$.  To do this, we claim it is sufficient to consider the case where $E\subset\R^2$ is a subset of the annulus with inner radius $1/2$ and outer radius 1, and can be decomposed into subsets $E_1,...,E_{k+1}$, each having positive measure with respect to a Frostman probability measure on $E$, which are "angle separated" in the sense that any $x_i\in E_i,x_j\in E_j$ with $i\neq j$ the angle between $x_i$ and $x_j$ is $\gtrsim 1$.  \\

Let $E\subset\R^2$ be compact with Hausdorff dimension greater than $2-\frac{1}{2k}$, and let $\mu$ be a Frostman probability measure with exponent $s>2-\frac{1}{2k}$ supported on $E$.  In particular, since $k\geq 1$ we have $s> \frac{3}{2}$.  We start by partitioning $E$ into dyadic annuli $A_j=\{x\in E:2^{j-1}\leq\|x\|\leq 2^{j}\}$.  We must have $\mu(A_j)>0$ for at least one $j$, and since the property of $\mathcal{A}_k(E)$ having positive measure is not affected by scaling we may assume without loss of generality that $j=0$.  We can replace $E$ with $A_0$ and rescale $\mu$ to obtain a new frostman probability measure, so without loss of generality we may assume $E$ is in the annulus we claimed.  \\

We further want to find angle separated subsets $E_1,...,E_{k+1}$ of positive measure.  We can do this by partitioning $[0,2\pi]$ into equal segments of length $\delta$ and considering the pieces of the annulus corresponding to this partition.  Considering alternating pieces, we can choose a subset of $E$ which has positive measure and where no two pieces of the annulus are consecutive.  Since $s>1$, if $\delta$ is small enough then each piece has measure $\leq\frac{1}{k+1}$, which means at least $k+1$ pieces must have positive measure.  We can let $E$ be the union of these $k+1$ pieces, again rescaling $\mu$ on each piece to obtain a probability measure where each of the $k+1$ pieces has measure $\frac{1}{k+1}$.  Our new set has the desired properties.  \\

To simplify notation, we will simply write integrals as $\int d\mu(x^i)$ or $\int d\mu^{k+1}(x)$.  In each case, the domain of integration is understood to be over $x_i\in E_i$. \\

The main idea of the proof is as follows.  We use $\mu$ to define a probability measure $\nu_k$ supported on $\mathcal{A}_k(E)$ and argue that it is absolutely continuous with respect to Lebesgue measure, which implies the result.  We reduce matters to bounding the integral

\[
\int\int d\mu^{k+1}(x)\:d\mu^{k+1}(y),
\]

where we are integrating over $x$ and $y$ such that $A(x)$ and $A(y)$ are (approximately) equal.  This allows us to apply the group action framework discussed in section 2, as well as Littlewood-Paley theory to write the above integral as a sum of integrals of the form

\[
\int\left(\prod_{j\in J} \int \mu_j(gx)\:d\mu(x)\right)dg,
\]

where $J$ is a finite set of $(k+1)$ indices.  We use fairly trivial bounds on $k-1$ of the factors, thus reducing to the $k=1$ case where we can use the mapping properties of generalized Radon transforms to obtain non-trivial estimates.  The result is a bound which is adequate when $s>2-\frac{1}{2k}$, hence the theorem.

\subsection{Proof of main theorem}

Let $A:(\R^2)^{k+1}\to \mathcal{A}_k$ be the map defined in section 2 and define a measure $\nu_k$ on $\mathcal{A}_k$ by

\[
\int_{\mathcal{A}_k} f(t)\:d\nu_k(t)=\int_{(\R^2)^{k+1}}f(A(x))\:d\mu^{k+1}(x).
\]

We want to prove $\mathcal{L}_{2k-1}(A_k(E))>0$.  Since $\nu_k$ is a probability measure supported on $\mathcal{A}_k(E)$, it suffices to prove that $\nu$ is absolutely continuous with respect to $\mathcal{L}_{2k-1}$.  Let $\f_\e$ be an approximation to the identity, and let $\nu_k^{\e}=\f_\e*\nu_k$.  We have

\[
\int_E \nu_k^{\e}(t)\:d\mathcal{L}_{2k-1}(t)\leq \mathcal{L}_{2k-1}(E)^{1/2}\|\nu_k^{\e}\|_{L^2},
\]

and the left hand side goes to $\nu(E)$ as $\e\to 0$.  So, a uniform bound on $\|\nu_k^{\e}\|_{L^2}$ implies $\nu$ is indeed absolutely continuous with respect to $\mathcal{L}_{2k-1}$.  Throughout, we abbreviate $d\mathcal{L}_{2k-1}(t)$ to simply $dt$.  We have

\begin{align*}
\nu_k^{\e}(t)&=\int\f_\e(t'-t)\:d\nu(t') \\
&=\int \f_\e(A(x)-t)d\mu^{k+1}(x) \\
&\approx \e^{-(2k-1)}\int_{\|A(x)-t\|<\e}d\mu^{k+1}(x) \\
\end{align*}

and therefore

\begin{align*}
\|\nu_k^{\e}\|_{L^2}^2 &\approx \e^{-2(2k-1)}\int\left(\int\cdots\int_{\substack{\|A(x)-t\|<\e \\ \|A(y)-t\|<\e}}d\mu^{k+1}(x)\:d\mu^{k+1}(y)\right)\:dt \\
&=\e^{-2(2k-1)}\int\cdots\int_{\|A(x)-A(y)\|<2\e}\left( \int_{\substack{\|A(x)-t\|<\e \\ \|A(x)-t\|<\e}} dt\right)d\mu^{k+1}(x)\:d\mu^{k+1}(y) \\
&\approx \e^{-(2k-1)}\int\cdots\int_{\|A(x)-A(y)\|<2\e}d\mu^{k+1}(x)\:d\mu^{k+1}(y).
\end{align*}

To estimate this integral, we decompose $d\mu^{k+1}(y)$ into Littlewood-Paley pieces.  Let $\psi$ be a Schwarz function supported in the range $\frac{1}{2}\leq |\x|\leq 4$ and constantly equal to 1 in the range $1\leq |\x|\leq 2$.  Define the $j$-th Littlewood-Paley piece of $\mu$ by $\widehat{\mu_j}(\x)=\psi(2^{-j}\x)\widehat{\mu}(\x)$.  The following bounds will be useful.

\begin{lem}
\label{bounds}
Let $\mu_j$ be as defined above.  Then,

\[
\|\mu_j\|_{L^\infty}\lesssim 2^{j(2-s)}
\]

and

\[
\|\mu_j\|_{L^2}\lesssim 2^{j(2-s)/2}.
\]
\end{lem}

We defer the proof to the next subsection to avoid distracting from the proof of the main theorem.  Using the Littlewood-Paley decomposition, we have

\[
\|\nu_k^{\e}\|_{L^2}^2 \approx \e^{-(2k-1)}\int\cdots\int_{\|A(x)-A(y)\|<2\e}\sum_{j_1>\cdots>j_{k+1}>0}\mu_{j_1}(y^1)\cdots \mu_{j_{k+1}}(y^{k+1})d\mu^{k+1}(x)\:dy.
\]

Here we have made a couple simple reductions.  First, restricting the sum to one where the indices are ordered only affects the value by a multiplicative constant, so we may impose a convenient order.  Second, we do not need to consider negative indices because summing over such indices clearly gives a bounded result.  This gives the expression above.  By Theorem \ref{gsub}, as $\e\to 0$ this has limit

\[
\approx \int\int\cdots\int \sum_{j_1>\cdots>j_{k+1}>0}\mu_{j_1}(gx^1)\cdots \mu_{j_{k+1}}(gx^{k+1})d\mu^{k+1}(x)\:dg. 
\]

By writing the integral in terms of the group action, we can separate each of the variables $x^i$ and bound this as a product of $k+1$ integrals, each of the form $\int \mu_j(gx)\:d\mu(x)$.  We bound $k-1$ of these factors using the $L^\infty$ bound in Lemma \ref{bounds}, and the fact that $\mu$ is a probability measure.  This shows that the above integral is

\[
\lesssim \sum_{j_1>\cdots>j_{k+1}}2^{j_3(2-s)}\cdots 2^{j_{k+1}(2-s)}\int\int\int \mu_{j_1}(gx^1)\mu_{j_2}(gx^2)\:d\mu(x^1)\:d\mu(x^2)\:dg.
\]

Running the sum in $j_3$ through $j_{k+1}$, we can reduce matters to the $k=1$ case as follows.

\begin{align*}
\|\nu_k^{\e}\|_{L^2}^2\lesssim &\sum_{j_1>j_2}2^{j_2(2-s)(k-1)}\int\int\int \mu_{j_1}(gx^1)\mu_{j_2}(gx^2)\:d\mu(x^1)\:d\mu(x^2)\:dg  \\
\approx &\e^{-1}\sum_{j_1>j_2}2^{j_2(2-s)(k-1)}\int\int\int\int_{|x^1\cdot x^{2\perp}-y^1\cdot y^{2\perp}|<\e} \mu_{j_1}(y^1)\mu_{j_2}(y^2)\:d\mu(x^1)\:d\mu(x^2)\:dy^1\:dy^2 \\
\approx &\e^{-2}\sum_{j_1>j_2}2^{j_2(2-s)(k-1)}\int\int\int\int\int_{\substack{|x^1\cdot x^{2\perp}-t|<\e \\ |y^1\cdot y^{2\perp}-t|<\e}}\mu_{j_1}(y^1)\mu_{j_2}(y^2)\:d\mu(x^1)\:d\mu(x^2)\:dy^1\:dy^2\:dt \\
=&\sum_{j_1>j_2}2^{j_2(2-s)(k-1)}\int\left(\e^{-1}\int\int_{|x^1\cdot x^{2\perp}-t|<\e}d\mu(x^1)\:d\mu(x^2)\right)\left(\e^{-1}\int\int_{|y^1\cdot y^{2\perp}-t|<\e}\mu_{j_1}(y^1)\mu_{j_2}(y^2)\:dy^1\:dy^2\right)\:dt \\
\approx & \sum_{j_1>j_2}2^{j_2(2-s)(k-1)}\int \nu_1^\e(t)\left\langle \mu_{j_1},\mathcal{R}_t^\e\mu_{j_2}\right\rangle\:dt 
\end{align*}

Here, $\left\langle \cdot,\cdot\right\rangle$ denotes the $L^2(\R^2)$ inner product and $\mathcal{R}_t^\e$ is the $\e$-approximation to the generalized Radon transform given by

\[
\mathcal{R}_tf(x)=\int_{x\cdot y^\perp=t} f(y)\eta(x,y)\:d\sigma_{x,t}(y),
\]

where $\sigma_{x,t}$ denotes the one dimensional Lebesgue measure on the line $\{y:x\cdot y^\perp=t\}$ and $\eta$ is a smooth cutoff function supported on the region $\frac{1}{2}\leq \|x\|,\|y\|\leq 1$.  This operator satisfies the following bounds.

\begin{lem}
\label{GRT}
Let $\mathcal{R}_t^\e$ and $\mu_j$ be as above.  Then

\[
\|\mathcal{R}_t^\e\mu_j\|_{L^2}\lesssim 2^{-j/2}2^{j(2-s)/2}.
\]

Also, if $|j-l|>5$ then

\[
\left\langle \mu_j,\mathcal{R}_t^\e\mu_l\right\rangle \lesssim 2^{-98\max(j,l)}
\]

\end{lem}

Again, we defer the proof to the next subsection.  Let

\[
S=\sum_{\substack{j_1,j_2 \\ 0<j_2\leq j_1}}2^{j_2(2-s)(k-1)}\sup_t \left\langle \mu_{j_1},\mathcal{R}_t^\e\mu_{j_2}\right\rangle.
\]

where $C$ is an appropriate constant.  We have proved

\[
\|\nu_k^\e\|_{L^2}^2\lesssim S\int\nu_1^\e(t)\:dt\lesssim S\|\nu_1^\e\|_{L^2}.
\]

Suppose $S$ is finite.  Then the $k=1$ case of the above inequality proves that $\|\nu_1^\e\|_{L^2}$ is bounded independent of $\e$, and this in turn implies the same for all $k>1$.  So, it suffices to prove $S$ is finite.  We break $S$ into two parts.  Let $S_1$ be the sum over indices $j_1-j_2\leq 5$ and $S_2$ the sum over indices $j_1>j_2+5$.  By the second part of Lemma \ref{GRT}, $S_2$ clearly converges.  We also have

\begin{align*}
S_1&\approx \sum_j 2^{j(2-s)(k-1)}\sup_t \left\langle \mu_{j},\mathcal{R}_t^\e\mu_{j}\right\rangle \\
&\lesssim \sum_j 2^{j(2-s)(k-1)}\|\mu_j\|_{L^2}\|\mathcal{R}_t^\e\mu_{j}\|_{L^2} \\
&\lesssim \sum_j 2^{j(2-s)(k-1)-\frac{j}{2}}\|\mu_j\|_{L^2}^2 \\
&\lesssim \sum_j 2^{j(2-s)(k-1)-\frac{j}{2}+j(2-s)},
\end{align*}

using Cauchy-Schwarz and Lemmas \ref{bounds} and \ref{GRT}.  If $s>2-\frac{1}{2k}$ then this sum converges as well.  

\subsection{Proofs of Lemmas \ref{bounds} and \ref{GRT}}

To prove the theorem, all that remains is to prove Lemmas \ref{bounds} and \ref{GRT} which were introduced in the previous section.  We restate and prove them here.

\begin{replem}{bounds}
Let $\mu_j$ be as defined above.  Then,

\[
\|\mu_j\|_{L^\infty}\lesssim 2^{j(2-s)}
\]

and

\[
\|\mu_j\|_{L^2}\lesssim 2^{j(2-s)/2}.
\]
\end{replem}

\begin{proof}
To prove the $L^\infty$ bound, we first observe $\mu_j(x)=2^{2j}\check{\psi}(2^j\cdot)*\mu(x)$.  Since $\psi$ is a Schwarz function, it satisfies $\psi(x)\lesssim (1+\|x\|)^{-2}$.  Therefore,

\[
|\mu_j(x)|\lesssim 2^{2j}\int (1+2^j\|x-y\|)^{-2}\:d\mu(y)
\]

We separate the integral into the range where $2^j\|x-y\|<1$ and $2^j\|x-y\|>1$.  In the first range, we have

\begin{align*}
&2^{2j}\int_{2^j\|x-y\|<1} (1+2^j\|x-y\|)^{-2}\:d\mu(y)  \\
\lesssim &2^{2j} \mu(\{y:2^j\|x-y\|<1\}) \\
\lesssim &2^{j(2-s)}
\end{align*}

In the second, we have

\begin{align*}
&2^{2j}\int_{2^j\|x-y\|>1} (1+2^j\|x-y\|)^{-2}\:d\mu(y)  \\
=&2^{2j}\sum_{m=0}^\infty\int_{2^m\leq 2^j\|x-y\|\leq 2^{m+1}} (1+2^j\|x-y\|)^{-2}\:d\mu(y)  \\
\lesssim & 2^{2j}\sum_{m=0}^\infty 2^{2m}\mu(\{y:2^m\leq 2^j\|x-y\|\leq 2^{m+1}\}) \\
\lesssim & 2^{j(2-s)}\sum_{m=0}^\infty 2^{m(s-2)}\\
\lesssim & 2^{j(2-s)}
\end{align*}

as claimed.  To prove the $L^2$ bound, we have

\[
\|\mu_j\|_{L^2}^2=\|\widehat{\mu_j}\|_{L^2}^2\lesssim \int_{\|\x\|\leq 2^{j+2}}|\widehat{\mu}(\x)|^2\:d\x\lesssim 2^{j(2-s)}.
\]

The last inequality comes from the Fourier representation of the energy integral of $\mu$; see \cite{WolffNotes}, proposition 8.5.
\end{proof}

\begin{replem}{GRT}

Let $\mathcal{R}_t^\e$ and $\mu_j$ be as above.  Then

\[
\|\mathcal{R}_t^\e\mu_j\|_{L^2}\lesssim 2^{-j/2}2^{j(2-s)/2}.
\]

Also, if $|j-l|>5$ then

\[
\left\langle \mu_j,\mathcal{R}_t^\e\mu_l\right\rangle \lesssim 2^{-98\max(j,l)}
\]

\end{replem}

\begin{proof}
The lemma is proved by a computation similar to those in \cite{EIT} and \cite{Incidence}; the idea is to show that $\widehat{\mathcal{R}_t^\e}$ decays rapidly outside the support of $\widehat{\mu_j}$ and use Plancherel.  Recall that $\f_\e$ is our approximation to the identity, so we have

\[
\mathcal{R}_t^\e\mu_j(x)=\int \mu_j(y)\eta(x,y)\f_\e(x\cdot y^\perp-t)\:dy.
\]

Using Fourier inversion on both $\mu_j$ and $\f_\e$, this is

\[
\mathcal{R}_t^\e\mu_j(x)=\int\int\int e^{2\pi i\x\cdot y}e^{2\pi i \tau(x\cdot y^\perp-t)}\widehat{\mu_j}(\x)\widehat{\f}(\e\tau)\eta(x,y)\:dy\:d\tau\:d\x,
\]

and therefore

\begin{align*}
\widehat{\mathcal{R}_t^\e\mu_j}(\zeta)&=\int\int\int\int e^{-2\pi i \zeta\cdot x}e^{2\pi i\x\cdot y}e^{2\pi i \tau(x\cdot y^\perp-t)}\widehat{\mu_j}(\x)\widehat{\f}(\e\tau)\eta(x,y)\:dx\:dy\:d\tau\:d\x. \\
&= \int\int \widehat{\mu_j}(\x)\widehat{\f}(\e\tau) I(\zeta,\x,\tau)\:d\tau\:d\x,
\end{align*}

where

\[
I(\zeta,\x,\tau)=\int\int e^{-2\pi i \zeta\cdot x}e^{2\pi i\x\cdot y}e^{2\pi i \tau(x\cdot y^\perp-t)}\eta(x,y)\:dx\:dy.
\]

We claim that $I(\zeta,\x,\tau)$ is negligible when $\zeta$ and $\x$ are in sufficiently separated dyadic annuli, as it is an oscillatory integral with phase function

\[
\Phi_{\zeta,\x,\tau}(x,y)=-\zeta\cdot x+\x\cdot y+\tau(x\cdot y^\perp-t)
\]

and we have

\[
\nabla \Phi_{\zeta,\x,\tau}(x,y)= (-\zeta+\tau y^\perp,\x-\tau x^\perp).
\]

Thus, critical points of $\Phi$ satisfy $\tau(x,y)=(\x^\perp,\zeta^\perp)$.  Suppose $2^{j-2}\leq |\zeta|\leq 2^{j+2}$ and $2^{l-2}\leq |\zeta|\leq 2^{l+2}$; without loss of generality, assume $l\leq j$.  Recall $\eta(x,y)$ is supported in the region $\frac{1}{2}\leq \|x\|,\|y\|\leq 1$; if that region contains a critical point, then we must have $|\tau|\leq 2^{l+3}$ and $|\tau|\geq 2^{j-2}$.  If $j-l>5$ this is not possible, so by nonstationary phase (for example \cite{WolffNotes}, proposition 6.1) we have $I(\zeta,\x,\tau)\lesssim_N 2^{-Nj}$ for any N when $\zeta,\x$ are in the range given above.  Without the assumption $j>l$, this becomes $I(\zeta,\x,\tau)\lesssim_N 2^{-N\max(j,l)}$ whenever $|j-l|>5$.  It follows that if $2^{l-2}\leq \|\zeta\|\leq 2^{l-2}$ with $|j-l|>5$, we have

\[
|\widehat{\mathcal{R}_t^\e\mu_j}(\zeta)|\lesssim_N \int\int|\widehat{\mu_j}(\x)||\widehat{\f}(\e\tau)|2^{-Nj}\:d\x\:d\tau
\]

Using the fact $|\mu_j|\leq 1$ with support of measure $\lesssim 2^{2j}$ and observing $\f$ is Schwarz, this gives 

\[
|\widehat{\mathcal{R}_t^\e\mu_j}(\zeta)|\lesssim 2^{-100\max(j,l)} \ \ \ \text{when}\ \ \  |j-l|>5
\]

With these preliminary computations done we are ready to prove the bounds we claimed in the lemma.  For the first, we have

\begin{align*}
\|\mathcal{R}_t^\e\mu_j\|_{L^2}^2 &=\|\widehat{\mathcal{R}_t^\e\mu_j}\|_{L^2}^2 \\
&=\int|\widehat{\mathcal{R}_t^\e\mu_j}(\x)|^2\:d\x \\
&=A+B,
\end{align*}

where $A$ is the integral over the range $2^{j-10}\leq \|\x\|\leq 2^{j+10}$ and $B$ is the remaining integral.  We have

\begin{align*}
A&=\int_{2^{j-3}\leq \|\x\|\leq 2^{j+3}}|\widehat{\mathcal{R}_t^\e\mu_j}(\x)|^2\:d\x \\
&\approx 2^{-j}\int_{2^{j-3}\leq \|\x\|\leq 2^{j+3}}\|\x\||\widehat{\mathcal{R}_t^\e\mu_j}(\x)|^2\:d\x \\
&\lesssim 2^{-j}\|\mathcal{R}_t^\e\mu_j\|_{L_{1/2}^2}^2
\end{align*}

The map $\mathcal{R}_t$ is a bounded linear operator $L^2\to L_{1/2}^2$ (see for example \cite{SS}, Chapter 8, Theorem 7.1).  It follows that $A\lesssim 2^j\|\mu_j\|_{L^2}^2$.  This gives the claimed bound in view of Lemma \ref{bounds}.  To bound $B$, we have

\[
B\approx \sum_{\substack{l \\ |j-l|>5}}\int_{2^{l-2}\leq \|\x\|\leq 2^{l+2}} |\mathcal{R}_t^\e\mu_j(\x)|^2\:d\x.
\]

We split further into two terms according to whether $l<j-5$ or $l>j+5$.  For the first, we have

\begin{align*}
&\sum_{\substack{l \\ l<j-5}}\int_{2^{l-2}\leq \|\x\|\leq 2^{l+2}} |\mathcal{R}_t^\e\mu_j(\x)|^2\:d\x \\
\lesssim & \sum_{\substack{l \\ l<j-5}}2^{2j}2^{-100j} \\
\lesssim &2^{-50j}.
\end{align*}

For the second,

\begin{align*}
&\sum_{\substack{l \\ l>j+5}}\int_{2^{l-2}\leq \|\x\|\leq 2^{l+2}} |\mathcal{R}_t^\e\mu_j(\x)|^2\:d\x \\
\lesssim & \sum_{\substack{l \\ l>j+5}}2^{2l}2^{-100l} \\
\lesssim &2^{-50j}
\end{align*}

This is clearly enough to have $B\lesssim 2^{-j}2^{j(s-2)}$ as claimed.  For the second bound in the lemma, suppose $|j-l|>5$.  We have

\begin{align*}
|\left\langle \mu_j,\mathcal{R}_t^\e\mu_l\right\rangle| &=|\left\langle \widehat{\mu_j},\widehat{\mathcal{R}_t^\e\mu_l}\right\rangle| \\
&=\left|\int \widehat{\mu_j}(\x)\widehat{\mathcal{R}_t^\e\mu_l}(\x)\:d\x\right| \\
&\lesssim 2^{-100\max(j,l)}\int |\widehat{\mu_j}(\x)|\:d\x \\
&\lesssim 2^{-100\max(j,l)}2^{2j} \\
&\leq 2^{-100\max(j,l)}2^{2\max(j,l)} \\
&=2^{-98\max(j,l)}.
\end{align*}

\end{proof}

This completes the proofs of both lemmas in section 3.2, and hence completes the proof of the main theorem.

\section{Proof of Theorem \ref{Sharpness}}

We conclude by proving Theorem \ref{Sharpness}.  Let $\Lambda_{q,s}$ be the $q^{-2/s}$-neighborhood of $\Z^2\cap\frac{1}{q}([q/2,q]\times [0,q])$, the right half of the lattice in the unit square with spacing $\frac{1}{q}$.  By Theorem 8.15 in \cite{Fractals}, $\cap_n \Lambda_{q_n,s}$ has dimension $s$ if $q_n$ is increasing sufficiently rapidly.  So, for large $q$, $\Lambda_{q,s}$ is an approximation to a set of dimension $s$.  We want to modify this example to fit our problem.  We use the following lemma.

\begin{lem}[\cite{Fractals}, Lemma 1.8]
Let $\psi:E\to F$ be surjective and Lipschitz, and let $H^s$ be the $s$-dimensional Hausdorff measure.  Then $H^s(F)\lesssim H^s(E)$.

\end{lem}

In particular, under the assumptions of the lemma, $\dim F\leq \dim E$.  If $\phi$ is bijective and Lipschitz in both directions, then $\dim F=\dim E$.  Let $E_{q,s}$ be the image of $\Lambda_{q,s}$ under the map

\[
\psi(x,y)=xe^{i\pi y/2}.
\]

It is clear that this map is injective on $[1/2,1]\times[0,1]$ and therefore bijective as a map $\Lambda_{q,s}\to E_{q,s}$.  Fix a rapidly increasing sequence $q_n$ and let $\Lambda_s=\cap_n \Lambda_{q_n,s},E_s=\cap_n E_{q_n,s}$.  By the lemma, $\dim E_s=s$.  It remains to prove $\mathcal{L}^{2k-1}(A_k(E_s))=0$. \\

We begin by counting the number of area types determined by the image of $\Z^2\cap([q/2,q]\times [0,q])$ under $\psi$, i.e. the set

\[
\{r(\cos \frac{\pi k}{2q},\sin\frac{\pi k}{2q}):r\in\Z\cap[q/2,q],k\in\Z\cap[0,q]\}.
\]

Let $\mathcal{C}$ be a complete set of representatives of the above set under the equivalence relation given by having the same area type.  Let $T$ be the operation defined on $\mathcal{C}$ where $Tx$ is obtained from $x$ by rotating clockwise until one $x^i$ is on the $x$-axis, and all others are still in the first quadrant.  It is clear that $T$ is injective on $\mathcal{C}$, so it is enough to bound $T(\mathcal{C})$.  Since every element of $T(\mathcal{C})$ has one point on the $x$-axis, there are $\approx q$ choices for that point and $\leq q^2$ choices for all other points.  So, $|T(\mathcal{C})|\lesssim q^{2k+1}$.  It follows that

\[
\mathcal{L}_{2k+1}(A_k(E_{q,s}))\lesssim (q^{-2/s})^{2k-1}q^{2k+1}.
\]

This tends to 0 as $q\to\infty$ provided $s<2-\frac{4}{2k+1}$, as claimed.

\end{document}